\documentclass[a4paper,11pt]{article}
\usepackage{amssymb}
\usepackage{amsfonts}
\usepackage[english]{babel}
\usepackage{amsmath,amscd}
\usepackage{amsthm}
\usepackage{a4wide}
\usepackage{color}
\usepackage{mathrsfs}

\title{On the Automorphisms of Quantum Weyl Algebras}
\begin{document}
\author{Andrew P. Kitchin\thanks{Andrew P. Kitchin thanks EPSRC for its support.}  ~and St\'ephane Launois}
\date{ }

\maketitle

\theoremstyle{theorem}
\newtheorem{theorem}{Theorem}[section]
\newtheorem{definition}[theorem]{Definition}
\newtheorem{lemma}[theorem]{Lemma}
\newtheorem{example}[theorem]{Example}
\newtheorem{conjecture}[theorem]{Conjecture}
\newtheorem{question}[theorem]{Question}
\newtheorem{proposition}[theorem]{Proposition}
\newtheorem{corollary}[theorem]{Corollary}
\newtheorem{remark}[theorem]{Remark}
\begin{abstract} 
Motivated by Weyl algebra analogues of the Jacobian conjecture and the tame generators problem, we prove quantum versions of these problems for a family of analogues to the Weyl algebras. In particular, our results cover the Weyl-Hayashi algebras and tensor powers of a quantization of the first Weyl algebra which arises as a primitive factor algebra of $U_q^+(\mathfrak{so}_5)$.  
\end{abstract}


\section{Introduction}\label{A1}
For a field $\Bbbk$, let $A:=\Bbbk\langle x_1,\ldots,x_n \rangle$ be the free associative algebra in $n$ variables and denote by $\mathrm{Aut}A$ the automorphism group of $A$. An automorphism $\psi$ of $A$ is called \textit{elementary} if it is of the form
$$\psi(x_1,\ldots,x_n)\mapsto (x_1,\ldots,x_{i-1},\alpha x_i+F,x_{i+1},\ldots,x_n),$$
where $\alpha\in\Bbbk^*:=\Bbbk\setminus\{0\}$ and $F\in\Bbbk\langle x_1,\ldots,x_{i-1},x_{i+1},\ldots,x_{n} \rangle$. The subgroup of $\mathrm{Aut}A$ generated by the elementary automorphisms is called the \textit{tame subgroup}, and an element of this subgroup is called \textit{tame}. An automorphism of $A$ not belonging to the tame subgroup is called \textit{wild}. Study of the automorphisms of $A$, and of factor algebras of $A$, has been ubiquitous over the last hundred years (for a comprehensive overview see \cite{Essen}). It was shown in \cite{Kulk} and \cite{Makar-Limanov} that the automorphisms groups of the polynomial ring and the free associative algebra in two variables are tame. Along with these results came the following natural problems:   
\begin{description}
\item 
P1\hspace{0.5mm}:~~Is every automorphism of the free associative algebra in $n$ variables tame?
\item
P2\hspace{0.5mm}:~~Is every automorphism of the commutative polynomial ring in $n$ variables tame?
\end{description}
Notably in \cite{Shes}, Nagata's automorphism, proposed in \cite{Nagata}, was shown to be wild yielding a negative answer to P2 in three variables. In \cite{Umirbaev} the Anick automorphism of the free associative algebra in three variables was shown to be wild also giving a negative answer to P1. Both Nagata's automorphism and the Anick automorphism are stably tame (see \cite{Smith}), thus, lifting either automorphism to higher order spaces unfortunately does not produce further wild automorphisms. To the best of the authors' knowledge the tame generators problems P1 and P2 remain unsolved for $n$ greater than $3$ generators. 

In his foundational paper on the Weyl algebra \cite{Dixmier}, Dixmier showed that every automorphism of the first Weyl algebra is tame. Given that the $n^{\mathrm{th}}$ Weyl algebra can be realised as a factor algebra of the free associative algebra in $2n$ variables, the natural Weyl algebra analogue of the tame generators problem follows:   
\begin{description}
\item 
P3\hspace{0.5mm}:~~Is every automorphism of the $n^{\mathrm{th}}$ Weyl algebra tame?  
\end{description}
Again, to the best of the authors' knowledge P3 remains unsolved for $n$ greater than one. Given the existence of wild automorphism in the polynomial ring and free algebra cases, one might suspect that a similar result will follow for higher order Weyl algebras.

In \cite{Launois}, primitive factor algebras of Gelfand-Kirillov dimension 2 of the positive part of the quantized enveloping algebra $U_q(\mathfrak{so}_5)$ were classified. These can be thought of as quantum analogues of the first Weyl algebra. Among those are the algebras $\mathscr{A}_{\alpha,q}$ with $\alpha\in\Bbbk^*$, where $\mathscr{A}_{\alpha,q}$ is the associative algebra in three variables $e_1, e_2,$ and $e_3$, subject to the following commutation relations:
\begin{align*}
& e_1 e_3=q^{-1}e_3 e_1,\\
& e_2 e_3=q e_3 e_2+\alpha,\\
& e_2e_1=q^{-1} e_1 e_2 -q^{-1}e_3,\\
& e_3^2+(q^2-1) e_3 e_1 e_2 +\alpha q(q+1) e_1=0.
\end{align*}
Setting $q=1$ and $\alpha=1$ we indeed get an algebra isomorphic to the first Weyl algebra. In \cite{Launois} these algebras are denoted $A_{\alpha,0}$ and for simplicity, we replace  $q^2$ with $q$. 

Let $\mathscr{H}_{q}^t$ denote the associative algebra with generators $\Omega, \Omega^{-1}, \Psi$ and $\Psi^{\dag}$ subject to the relations 
\begin{align}\label{Hayashi}
& \Omega\Omega^{-1}=\Omega^{-1}\Omega=1,\nonumber\\
& \Psi \Omega=q\Omega\Psi,~~ \Psi^{\dag} \Omega=q^{-1} \Omega \Psi^{\dag},\nonumber\\
&\Psi \Psi^{\dag}=\frac{q^t \Omega^t-q^{-t}\Omega^{-t}}{q^t-q^{-t}},~\mathrm{and}~~ \Psi^{\dag}\Psi=\frac{ \Omega^t-\Omega^{-t}}{q^t-q^{-t}}.
\end{align}
By setting $t=1$ we retrieve the Weyl Hayashi algebra $\mathscr{H}_{q}^1$ studied in \cite{Alev} and \cite{Kirkman}. When $t=2$ we get the original algebras introduced by Hayashi in \cite{Hayashi}. In this article we will consider the generalization $\mathscr{H}_{q}^t$ which covers both conventions. In \cite{Hayashi} Hayashi introduced $\mathscr{H}_{q}^2$ as $q$-analogue of the Weyl algebra to construct oscillator representations of quantum enveloping algebras. In \cite{Kirkman} it was shown that the algebras $\mathscr{H}_{q}^1$ arise as factor algebras of a $q$-analogue of the universal enveloping algebra of the Heisenberg Lie algebra. It was also shown in (\cite{Launois}, Section 3) that $\mathscr{H}_{q}^1$ appears as factor algebras of the positive part of the quantized enveloping algebra $U_q(\mathfrak{so}_5)$. 

The tame generators problems, and in particular P3, makes it natural to consider if the complexity of the automorphism group of quantum analogues of the $n^{\mathrm{th}}$ Weyl algebra fundamentally changes as $n$ increases. In this article we arrive at analogues to the $n^{\mathrm{th}}$ Weyl algebra by taking the tensor product (over the ground field) of $n$ copies of our first Weyl algebra analogues $\mathscr{A}_{\alpha,q}$ and $\mathscr{H}_{q}^t$. By showing that these algebras are part of a family of generalized Weyl algebras that we call {\em quantum Weyl analogue (qwa) algebras}, we are able to define the notion of {\em qwa-tame} (see Section \ref{qtgp} and specifically Definition \ref{def1}). Using our definition we show that the automorphism groups of our analogues are well behaved as we increase the number of tensor copies. Precisely we prove the following quantum analogues to the tame generator problem:
\begin{theorem}\label{A2}
Every automorphism of $\mathscr{A}_{\alpha,q}\otimes\cdots \otimes \mathscr{A}_{\alpha,q}$ is qwa-tame for $\alpha\in\Bbbk^*$ and $q\in \Bbbk^*\setminus\{z|z^2= 1 \}$. 
\end{theorem}
\begin{theorem}\label{A3}
Every automorphism of $\mathscr{H}_{q}^t\otimes\cdots \otimes \mathscr{H}_{q}^t$ is qwa-tame for $\alpha\in\Bbbk^*$ and $q\in \Bbbk^*\setminus\{z|z^{2t}= 1 \}$.          
\end{theorem}
In general computing the automorphism group of an algebra can be very difficult. Recently some progress has been made to produce a uniform approach to this problem for a large class of algebras (see \cite{Ceken1}). In \cite{Ceken2} the same authors use their approach to show that the automorphism group of tensor products of the so-called q-quantum Weyl algebra is tame. Theorem \ref{A2} and Theorem \ref{A3} can be seen as a direct analogue to \cite[Theorem 2]{Ceken2}.    

Dixmier also made the now famous conjecture: Every endomorphism of the $n^{\mathrm{th}}$ Weyl algebra is an automorphism. Tsuchimoto, \cite{Tsuchimoto}, and Belov-Kanel and Kontsevich, \cite{Belov}, proved independently that the Dixmier conjecture is stably equivalent to the Jacobian conjecture of Keller \cite{Keller}. It is natural to ask Dixmier's question for related algebras (see \cite{Bavula3,Richard}), and especially generalizations and quantizations of the Weyl algebras (see \cite{Backelin,Lopes}). In \cite{Launois2}, every endomorphism of $\mathscr{A}_{\alpha,q}$ (and more generally simple quantum generalized Weyl algebras), when $q$ is not a root of unity, was shown to be an automorphism. In this article we show that every homomorphism between two of our analogues of $n^{\mathrm{th}}$ Weyl algebra is invertible. Precisely, we prove the following theorems:
\begin{theorem}\label{Them1}
If  $q$ is not a root of unity and $\alpha_i,\widetilde{\alpha}_i\in\Bbbk^*$ for $i\in\{1,\ldots,n\}$, then every homomorphism between $\mathscr{A}_{\alpha_1,q}\otimes\cdots \otimes \mathscr{A}_{\alpha_n,q}~~\mathrm{and}~~\mathscr{A}_{\widetilde{\alpha}_1,q}\otimes\cdots \otimes \mathscr{A}_{\widetilde{\alpha}_n,q}$ is invertible. 
\end{theorem}
\begin{theorem}\label{Them2k}
If $q$ is a non root of unity, then every endomorphism of $\mathscr{H}_{q}^t\otimes\cdots \otimes \mathscr{H}_{q}^t$ is an automorphism.
\end{theorem} 

In parallel to the pathology often encountered when considering algebras over nonzero characteristic fields, in quantum algebra, considering quantizations at roots of unity can be equally problematic. Given the current interest in reduction modulo $p$ techniques and results in the context of differential operators (see for instance \cite{Belov2,Kontsevich}), it is natural to extend the work in \cite{Launois2} to study the endomorphisms of quantum generalized Weyl algebras when $q$ is a root of unity. This case can be thought of as the quantum analogue of reduction modulo $p$ (see for instance \cite{Backelin}). Thus, we extend the classification of endomorphisms used in the proof of Theorem 1.1 of \cite{Launois2}, to include the case where $q$ is a root of unity other than $\pm 1$. We show that there exist non-invertible endomorphisms in this case (see Corollary \ref{endi}).   

\section{Preliminaries}\label{Prel}
To prove our quantum analogues of the tame generators problem and the Dixmier conjecture, we will exploit that the algebras $\mathscr{A}_{\alpha,q}$ and $\mathscr{H}_{q}^t$ are isomorphic to generalized Weyl algebras of degree 1. Our strategy will then be to classify the homomorphisms between tensor products of these algebras. 

Recall that for a $\Bbbk$-algebra $R$, a ($\Bbbk$-algebra) automorphism $\sigma$ of $R$, and a central element of $R,$ say $a$, the generalized Weyl algebra $R(\sigma,a)$ of degree 1 is the algebra extension of $R$ by the two indeterminates $x$ and $y$ subject to the relations 
$$xy=\sigma(a),~~yx=a,~~xr=\sigma(r)x,~\mathrm{and}~~yr=\sigma^{-1}(r)y~~\mathrm{for~all}~r\in R.$$
The isomorphisms and automorphisms of generalized Weyl algebras of degree 1 have been widely examined (see \cite{Bavula1,Richard1,Vivas}).
For $d \in\mathbb{N}^*$, $q\in\Bbbk^*\setminus \{z|z^d= 1 \}$ and $\sigma\in \mathrm{Aut}(\Bbbk[h^{\pm 1}])$ such that $\sigma(h)=qh$, we denote by $A(d,q)$ the generalized Weyl algebra $\Bbbk[h^{\pm 1}](\sigma, h^d-1)$. Using Proposition 3.10 of \cite{Launois} and Theorem A of \cite{Vivas} we have that $ \mathscr{A}_{\alpha,q}\simeq A(1,q)$.

\begin{remark}\label{rmk:utile}
Given that $ \mathscr{A}_{\alpha,q}\simeq A(1,q)$ for all $\alpha \in \Bbbk^*$, Theorem \ref{Them1} reduces to proving that every endomorphism of $A(1,q)\otimes\cdots \otimes A(1,q)$ is an automorphism when $q$ is not a root of unity.
\end{remark}

 By the isomorphism which sends
\begin{align*}
&\Omega\mapsto h,~~\Psi\mapsto \frac{q^{-t}h^{-t}}{q^{t}-q^{-t}}x,~\mathrm{and}~~\Psi^{\dag}\mapsto y,
\end{align*}
we have that $\mathscr{H}_{q}^t \simeq A(2t,q)$. Since the algebras $\mathscr{A}_{\alpha,q}$ and $\mathscr{H}_{q}^t$  are analogues of the first Weyl algebra, we can produce analogues, and in the case of $\mathscr{A}_{\alpha,q}$ a quantization, of the $n^{\mathrm{th}}$ Weyl algebra by taking a tensor product, over $\Bbbk$, of $n$ copies of the original algebra. Thus, for $n,d \in\mathbb{N}^*$, $\textbf{q}:=(q_1,\ldots, q_n)\in(\Bbbk^*\setminus\{z|z^t= 1 \})^n$ and $\sigma_i\in \mathrm{Aut}(\Bbbk[h_i^{\pm 1}])$ such that $\sigma_i(h_i)=q_ih_i$, we define the quantum Weyl analogue (qwa) algebras 
$$A(n,d,\textbf{q}):=\displaystyle\bigotimes_{i=1}^n \Bbbk[h_i^{\pm 1}](\sigma_i, h_i^d-1).$$
By extending the above isomorphisms we can realize the algebras $$\mathscr{A}_{\alpha,q_1}\otimes\cdots \otimes \mathscr{A}_{\alpha,q_n}~~\mathrm{and}~~ \mathscr{H}_{q_1,t}\otimes\cdots \otimes \mathscr{H}_{q_n,t}$$ as members of the family of algebras $A(n,d,\textbf{q})$. 

Since the category of generalized Weyl algebras is closed under tensor product, $A(n,d,\textbf{q})$ is a degree $n$ generalized Weyl algebra in the sense of \cite{Bavula2}. For simplicity we fix the notation $N:=\{1,\ldots n \}$ and $a_d(h_i)=h_i^d-1$. Precisely, $A(n,d,\textbf{q})$ is the $\Bbbk$-algebra generated by $x_i,y_i,h_i$ and $h_i^{-1}$ subject to the relations 
\begin{align*}
x_{i} h_{i}=q_{i} h_{i} x_{i}, ~~ &y_{i} h_{i}=q_{i}^{-1}h_{i} y_{i},~~ x_{i} y_{i}=a_d(q_i h_i),~~ y_{i} x_{i}=a_d(h_i),~~ h_{i}^{\pm 1}h_{i}^{\mp 1}=1
\end{align*} 
and the commutation relations
\begin{align}\label{comm}
&h_{i}h_{j}=h_{j}h_{i}, ~~h_{i}x_{j}=x_{j}h_{i},~~h_{i}y_{j}=y_{j}h_{i},\nonumber\\
 &x_{i}x_{j}=x_{j}x_{i},~~x_{i}y_{j}=y_{j}x_{i},~\mathrm{and}~~y_{i}y_{j}=y_{j}y_{i}
\end{align}
for $i,j\in N$ and $i\neq j$.

The property that any degree $n$ generalized Weyl algebra is $\mathbb{Z}^n$-graded is integral to the proof of Theorems \ref{A2} and \ref{Them1}. Thus, we recall this grading from \cite{Bavula4} applying it to $A(n,d,\textbf{q})$. For a vector $\textbf{k}:=(k_1,\ldots,k_n)\in\mathbb{Z}^n$ we set $w_{\textbf{k}}:=w_{k_1}(1)\cdots w_{k_n}(n)$, where for $i\in N$ and $m\geq 0$ we have $$w_{m}(i)=x_{i}^m,~ w_{-m}(i)=y_{i}^m, ~\mathrm{and} ~~ w_0(i)=1 .$$ It follows from the relations of $A(n,d,\textbf{q})$ that 
\begin{equation}\label{Grading}
A(n,d,\textbf{q}):=\bigoplus_{\textbf{k}\in{\mathbb{Z}^n}}A_{(\textbf{k})}
\end{equation}
is a $\mathbb{Z}^n$-graded algebra, where $A_{(\textbf{k})}:= \Bbbk[h_{1}^{\pm 1},\ldots,h_{n}^{\pm 1}] w_{\textbf{k}}=w_{\textbf{k}}\Bbbk[h_{1}^{\pm 1},\ldots,h_{n}^{\pm 1}].$

When classifying automorphisms or isomorphisms, it can often be illuminating to consider normal elements, since normality is preserved by invertible homomorphisms. Indeed, this approach was adopted in \cite{Richard1} to classify, up to isomorphism, quantum generalized Weyl algebras over a polynomial ring. For general homomorphisms, normality is not preserved. Instead, we exploit that any homomorphism maps invertible elements to invertible elements. It is clear that the algebras $A(n,d,\textbf{q})$ have non-trivial units since the component generalized Weyl algebras are defined over Laurent polynomial rings. We will now state the classification of units for the algebras $A(n,d,\textbf{q})$.
\begin{lemma}\label{Unitsl}
Any unit in $A(n,d,\textbf{\emph{q}})$ is of the form $\gamma h_{1}^{m_{1}} \cdots h_{n}^{m_{n}},$ where $\gamma \in\Bbbk^*$ and $m_{1},\ldots,m_{n} \in\mathbb{Z}$.
\end{lemma}
The result is well known for the algebra $A(1,d,\textbf{q})$ (see for example \cite[Lemma 5.1 (i)]{Bavula1}). The proof for the general case follows.
\begin{proof}
Let $u$ be a unit in $A(n,d,\textbf{q})$ with inverse $u^{-1}$. Using that $A(n,d,\textbf{q})$ is $\mathbb{Z}^n$-graded (see Equation (\ref{Grading})) we write 
$$u=\displaystyle\sum_{\textbf{k}\in\mathbb{Z}^n} W_{\textbf{k}}~~ \mathrm{and} ~~u^{-1}=\displaystyle\sum_{\textbf{s}\in\mathbb{Z}^n} V_{\textbf{s}},$$
where $W_{\textbf{k}},V_{\textbf{k}}\in A_{(\textbf{k})}$ for all ${\textbf{k}}\in\mathbb{Z}^n$ (all but a finite number of them being equal to zero). Thus, we have
\begin{equation*}
\left(\displaystyle\sum_{\textbf{k}\in\mathbb{Z}^n} W_{\textbf{k}} \right) \left( \displaystyle\sum_{\textbf{s}\in\mathbb{Z}^n} V_{\textbf{s}} \right)=1.
\end{equation*}
Noting that $\mathbb{Z}^n$ is a totally ordered group and $1\in A_{(\textbf{0})}$, we find that $u=W_{\textbf{k}}$ and $u^{-1}=V_{-\textbf{k}}$ for some $\textbf{k}\in \mathbb{Z}^n$. Since $h^d-1$ is not invertible, we come to the conclusion $\textbf{k}=\textbf{0}$. Therefore any unit of $A(n,d,\textbf{q})$ is a unit of $\Bbbk[h_1^{\pm 1},\ldots, h_n^{\pm 1}]$. By the definition of $A(n,d,\textbf{q})$ we can see that this condition is sufficient.
\end{proof}

\section{Classification of Homomorphisms}\label{Secthe}
Before giving our classification of homomorphisms we introduce for simplicity the following notation. For $0<r\leq n$, let $R:=\{1,\ldots,r\}$, $\widetilde{\textbf{q}}:=(\widetilde{q}_1,\ldots,\widetilde{q}_r  )\in (\Bbbk^*\setminus\{z|z^d= 1 \})^r$, and $\widetilde{A}(r,d,\widetilde{\textbf{q}}):=A(r,d,\widetilde{\textbf{q}})$ distinguishing $\widetilde{A}(r,d,\widetilde{\textbf{q}})$ from $A(n,d,\textbf{q})$ by marking every generator and indeterminate of $\widetilde{A}(r,d,\widetilde{\textbf{q}})$ with a tilde (for example $\widetilde{h_i}$). We attach primes (for example $p_i'$) to expand our choice of notation, this is in noway related to the derivative of an element. Finally, we use the notation $\mathbb{Z}^*:=\mathbb{Z}\setminus\{0\}$ throughout. 

We will now classify the homomorphisms between $\widetilde{A}(r,d,\widetilde{\textbf{q}})$ and $A(n,d,\textbf{q})$. 

\begin{theorem}\label{corol1}
\begin{enumerate}
\item Let $\psi$ be a homomorphism from $\widetilde{A}(r,d,\widetilde{\emph{\textbf{q}}})$ to $A(n,d,\emph{\textbf{q}})$.
\begin{description}
\item (i)~~There exists a partial permutation $w:R\rightarrow N$, $(\tau_1,\ldots,\tau_r)\in\{0,1\}^r$ and $(m_{1},\ldots,m_r)\in(\mathbb{Z}^*)^r$ such that 
\begin{equation}\label{theoeq1}
q_{w(i)}^{(-1)^{\tau_i} m_{i}}=\widetilde{q}_{i}
\end{equation}
for $i\in R.$
\item (ii) There exists a matrix $(t_{i,j})\in\mathcal{M}_{r,n}(\mathbb{Z})$ such that
\begin{equation}\label{theoeq2}
q_{w(i)}^{t_{l,w(i)} (1-\tau_i)} q_{w(i)}^{-t_{l,w(i)}\tau_i} q_{w(l)}^{-t_{i,w(l)}(1-\tau_l)} q_{w(l)}^{t_{i,w(l)}\tau_l}=1
\end{equation}
for all $i,l \in R$.
\item (iii)~~ For $i\in R$ there exist $p_{i}(h_{w(i)}),p_{i}'(h_{w(i)})\in\Bbbk[h_{w(i)}^{\pm1}]$, and $\gamma_i\in\Bbbk^*$ such that 
\begin{equation}\label{theoeq3}
p_{i}(h_{w(i)})p_{i}'(h_{w(i)})a_d(q_{w(i)}^{(1-\tau_i)} h_{w(i)})=a_d(\widetilde{q_{i}}\gamma_i h_{w(i)}^{m_{i}}).
\end{equation}
\item (iv)~~The homomorphism $\psi$ is defined on the generators of $\widetilde{A}(r,d,\widetilde{\emph{\textbf{q}}})$ as follows:
\end{description}
\begin{itemize}
\item $\psi(\widetilde{h_{i}})=\gamma_i h_{w(i)}^{m_{i}},$ where $\gamma_i^d=\widetilde{q_{i}}^{-\tau_i d}$
\item $\psi(\widetilde{x_{i}})=p_{i}(h_{w(i)})b_i h_{1}^{t_{i,1}}\cdots h_{n}^{t_{i,n}} x_{w(i)}^{(1-\tau_i)}y_{w(i)}^{\tau_i}$, where $b_i\in\Bbbk^*$
\item $\psi(\widetilde{y_{i}})=x_{w(i)}^{\tau_i}y_{w(i)}^{(1-\tau_i)} p_{i}'(h_{w(i)}) b_i^{-1}h_{1}^{-t_{i,1}}\cdots h_{n}^{-t_{i,n}}$
\end{itemize} 

\item
Conversely, assume there exist a partial permutation $w:R\rightarrow N$, $(\tau_1,\ldots,\tau_r)\in\{0,1\}^r$, $(m_{1},\ldots,m_r)\in(\mathbb{Z}^*)^r$,$(b_1,\ldots,b_r),(\gamma_1,\ldots,\gamma_r)\in(\Bbbk^*)^r$, a matrix $(t_{i,j})\in \mathcal{M}_{r,n}(\mathbb{Z})$ and $(p_1,\ldots,p_r), (p_1',\ldots,p_r')\in(\Bbbk[h^{\pm 1}])^r$ such that Equations (\ref{theoeq1}), (\ref{theoeq2}) and (\ref{theoeq3}) hold, and $\gamma_i^d=\widetilde{q_{i}}^{-\tau_i d}$ for $i\in R$. Then, there exists a unique homomorphism $\psi_{\textbf{o}}$ (where $\textbf{o}$ encodes the information in the hypothesis) from $\widetilde{A}(r,d,\widetilde{\emph{\textbf{q}}})$ to $A(n,d,\emph{\textbf{q}})$ defined on the generators of $\widetilde{A}(r,d,\widetilde{\emph{\textbf{q}}})$ as follows:   
\begin{itemize}
\item $\psi_{\textbf{o}}(\widetilde{h_{i}})=\gamma_i h_{w(i)}^{m_{i}}$
\item $\psi_{\textbf{o}}(\widetilde{x_{i}}) =p_{i}(h_{w(i)})b_i h_{1}^{t_{i,1}}\cdots h_{n}^{t_{i,n}} x_{w(i)}^{(1-\tau_i)}y_{w(i)}^{\tau_i}$
\item $\psi_{\textbf{o}}(\widetilde{y_{i}}) =x_{w(i)}^{\tau_i}y_{w(i)}^{(1-\tau_i)} p_{i}'(h_{w(i)})b_i^{-1} h_{1}^{-t_{i,1}}\cdots h_{n}^{-t_{i,n}}$
\end{itemize}
\end{enumerate}
\end{theorem}

\begin{proof}
We dedicate the rest of Section \ref{Secthe} to the proof of Theorem \ref{corol1}. For ease of understanding we break down our proof into four steps, giving summaries at the beginning and end of each step. Steps 1-3 combine to prove statement 1 of Theorem \ref{corol1}, and Step 4 proves statement 2.    
\end{proof}

\subsection{Step 1}

\textit{In Step 1 we will determine, for $i\in R$, the action of a homomorphism $\psi$ on $\widetilde{h_{i}}$. We will also show that $\psi(\widetilde{x_{i}})\in A_{(\textbf{\emph{k}})}$ and $\psi(\widetilde{y_{i}})\in A_{(-\textbf{\emph{k}})}$, where $A_{(\textbf{\emph{k}})}$ and $A_{(-\textbf{\emph{k}})}$ are as defined in Equation (\ref{Grading}).} \\

Let $\psi$ be a homomorphism from $\widetilde{A}(r,d,\widetilde{\textbf{q}})$ to $A(n,d,\textbf{q})$. Since units are preserved by homomorphisms, from Lemma \ref{Unitsl} we deduce that, for all $i\in R,$ $$\psi(\widetilde{h_{i}})=\gamma_i h_{1}^{m_{i,1}}\cdots h_{n}^{m_{i,n}},$$ where $\gamma_i \in\Bbbk^*,$ and $m_{i,1},\ldots,m_{i,n} \in\mathbb{Z}$. 

We first prove that for all $i\in R$ there exists $l\in N$ such that $m_{i,l}\neq 0$. By contradiction assume $m_{i,1}=\ldots=m_{i,n}=0$. Applying $\psi$ to the relations $\widetilde{x_{i}}\widetilde{h_{i}}=\widetilde{q_{i}}\widetilde{h_{i}}\widetilde{x_{i}}$ and $\widetilde{y_{i}}\widetilde{h_{i}}=\widetilde{q_{i}}^{-1}\widetilde{h_{i}}\widetilde{y_{i}}$ we find that $\psi(\widetilde{x_{i}})=0=\psi(\widetilde{y_{i}})$. Applying $\psi$ to $\widetilde{y_{i}}\widetilde{x_{i}}=a_d(\widetilde{h_{i}})$ and $\widetilde{x_{i}}\widetilde{y_{i}}=a_d(\widetilde{q_{i}}\widetilde{h_{i}})$ gives us that
\begin{equation}\label{eqala}
a_d(\gamma_{i})=0=a_d(\widetilde{q_{i}} \gamma_{i})
\end{equation}
implying that $\widetilde{q}_i^{~d}=1$ contradicting our assumption. Hence for all $i\in R$ there exists $l\in N$ such that $m_{i,l}\neq 0$. 

Applying $\psi$ to the relation $\widetilde{x_{i}} \widetilde{y_{i}}=a_d (\widetilde{q_{i}} \widetilde{h_{i}})$ we get 
\begin{equation}\label{Eqak}
\psi(\widetilde{x_{i}}) \psi( \widetilde{y_{i}})=a_d(\widetilde{q_{i}} \gamma_i h_{1}^{m_{i,1}}\cdots h_{n}^{m_{i,n}}).
\end{equation}
Using that $A(n,d,\textbf{q})$ is $\mathbb{Z}^n$-graded (see Equation (\ref{Grading})) we write 
$$\psi(\widetilde{x_{i}})=\displaystyle\sum_{\textbf{k}\in\mathbb{Z}^n} W_{\textbf{k}}~~ \mathrm{and} ~~\psi(\widetilde{y_{i}})=\displaystyle\sum_{\textbf{s}\in\mathbb{Z}^n} W'_{\textbf{s}}$$
where $W_{\textbf{k}},W'_{\textbf{k}}\in A_{(\textbf{k})}$ for all ${\textbf{k}}\in\mathbb{Z}^n$ (and all but a finite number of them being equal to zero). Substituting these expressions into Equation (\ref{Eqak}) yields
\begin{equation}\label{AAA11}
\left(\displaystyle\sum_{\textbf{k}\in\mathbb{Z}^n} W_{\textbf{k}} \right) \left( \displaystyle\sum_{\textbf{s}\in\mathbb{Z}^n} W'_{\textbf{s}} \right)=a_d(\widetilde{q_{i}} \gamma_i h_{1}^{m_{i,1}}\cdots h_{n}^{m_{i,n}}).
\end{equation}
Noting that $a_d(\widetilde{q_{i}} \gamma_i h_{1}^{m_{i,1}}\cdots h_{n}^{m_{i,n}})\in A_{(\textbf{0})}$ we find that $\psi(\widetilde{x_{i}})=W_{\textbf{k}}$ and $\psi(\widetilde{y_{i}})=W'_{-\textbf{k}}$ for some $\textbf{k}\in \mathbb{Z}^n$. Up to reordering the tensor product factors in $A(n,d,\textbf{q})$, it suffices to only consider the case where $\textbf{k}=(k_1,\ldots,k_e,-k_{e+1},\ldots,-k_n)$ with $k_j \in \mathbb{Z}_{\geq 0}$ for $j\in N$. First consider the case where $\textbf{k}=\textbf{0}$. Thus, $\psi(\widetilde{x_{i}})=P_{i}(h_{1},\ldots,h_{n})$ for $P_{i}(h_{1},\ldots,h_{n})$ a Laurent polynomial in the variables $h_{1},\ldots,h_{n}$. Applying $\psi$ to the relation $\widetilde{x_{i}}\widetilde{h_{i}}=\widetilde{q_{i}}\widetilde{h_{i}}\widetilde{x_{i}}$ implies that $P_{i}(h_{1},\ldots,h_{n})=0$ since $\widetilde{q_{i}}\neq 1$. Now, applying $\psi$ to the relation $\widetilde{y_{i}}\widetilde{x_{i}}=a_d(\widetilde{h_{i}})$ gives us the contradiction $\psi(a_d(\widetilde{h_{i}}))=0$. Thus, there must be at least one nonzero entry in $\textbf{k}$.
We now have    
\begin{align*}
\psi(\widetilde{x_{i}})&=P_{i}(h_{1},\ldots,h_{n})x_{1}^{k_1}\cdots x_{e}^{k_e} y_{e+1}^{k_{e+1}}\cdots y_{n}^{k_{n}}\\
&\mathrm{and}~~\psi(\widetilde{y_{i}})=y_{1}^{k_1}\cdots y_{e}^{k_e}x_{e+1}^{k_{e+1}}\cdots x_{n}^{k_{n}}P'_{i}(h_{1},\ldots,h_{n})
\end{align*}
where $P_{i}(h_{1},\ldots,h_{n})$ and $P'_{i}(h_{1},\ldots,h_{n})$ are nonzero Laurent polynomials in the variables $h_{1},\ldots,h_{n}$. Thus, we can rewrite Equation (\ref{Eqak}) as
\begin{equation}\label{Eq4a}
\begin{aligned}
P_{i}(h_{1},\ldots,h_{n})x_{1}^{k_1}\cdots x_{e}^{k_e} y_{e+1}^{k_{e+1}}\cdots y_{n}^{k_{n}}y_{1}^{k_1}\cdots y_{e}^{k_e}x_{e+1}^{k_{e+1}}\cdots x_{n}^{k_{n}}P'_{i}(h_{1},\ldots,h_{n})&
\\=a_d(\widetilde{q_{i}} \gamma_i h_{1}^{m_{i,1}}\cdots h_{n}^{m_{i,n}})&.
\end{aligned}
\end{equation}
Standard manipulation (see \cite[Equation (5)]{Bavula1}) of Equation (\ref{Eq4a}) gives us that 
$$U_i(h_{1},\ldots,h_{n}) \left( \prod_{s=1}^{e}\prod_{l=1}^{k_s}a_d(q_{s}^{l} h_{s}) \right) \left(\prod_{s=e+1}^{n}\prod_{l=0}^{k_s-1}a_d(q_{s}^{-l} h_{s}) \right)=a_d (\widetilde{q_{i}} \gamma_i h_{1}^{m_{i,1}}\cdots h_{n}^{m_{i,n}})$$
where $U_i(h_{1},\ldots,h_{n})=P_{i}(h_{1},\ldots,h_{n})P'_{i}(h_{1},\ldots,h_{n})$. Using that $a_d(X)=X^d-1$ we get
\begin{equation}\label{prod1}
\begin{aligned}
U_i(h_{1},\ldots,h_{n}) \left( \prod_{s=1}^{e}\prod_{l=1}^{k_s} \left( q_{s}^{l d}  h_{s}^d-1\right) \right)& \left(\prod_{s=e+1}^{n}\prod_{l=0}^{k_s-1}\left( q_{s}^{-l d}  h_{s}^d-1\right) \right)\\
&=a_d (\widetilde{q_{i}} \gamma_i h_{1}^{m_{i,1}}\cdots h_{n}^{m_{i,n}}).
\end{aligned}
\end{equation}
Pick $j\in N$ such that $k_{j}\neq 0$. Evaluating Equation (\ref{prod1}) at $h_{j}=q_{j}^{-1}$ if $j\in \{1,\ldots,e\}$, or $h_{j}=1$ if $j\in \{e+1,\ldots,n\}$, implies that $m_{i,s}=0$ for all $s\in N\setminus\{j\}$. We cannot repeat this process by evaluating at an alternate zero, since we have proved that at least one $m_{i,l}\neq 0$ with $l\in N$. Hence for each $i\in R$ there exists a unique $j\in N$ such that $m_{i,j}\neq 0$. Moreover $\psi(\widetilde{h_{i}})=\gamma_i h_{j}^{m_{i,j}}$. We set $w(i):=j$, but suppress this notation for simplicity until Step 3, where we show that the map $w: R\rightarrow N$ is a partial permutation. Since the double subscript is now redundant we simplify our notation and set $m_{i}:=m_{i,j}$. 

\textit{To summarize, in Step 1 we have shown that, for all $i \in R$, 
\begin{align*}
\psi(\widetilde{h_{i}})=\gamma_i h_{j}^{m_{i}},~~\psi(\widetilde{x_{i}})=P_{i}(h_{1},\ldots,h_{n})x_{1}^{k_1}\cdots x_{e}^{k_e} y_{e+1}^{k_{e+1}}\cdots y_{n}^{k_{n}},\\
\mathrm{and}~~\psi(\widetilde{y_{i}})=y_{1}^{k_1}\cdots y_{e}^{k_e}x_{e+1}^{k_{e+1}}\cdots x_{n}^{k_{n}}P'_{i}(h_{1},\ldots,h_{n})
\end{align*}
where $j=w(i)\in N$, $\gamma_i\in \Bbbk^*$, $m_i\in\mathbb{Z}$ and $P_{i}(h_{1},\ldots,h_{n}),P'_{i}(h_{1},\ldots,h_{n})\in\Bbbk[h_1^{\pm 1},\ldots,h_n^{\pm 1}]$ and $(k_1,\ldots,k_n) \in (\mathbb{Z}_{\geq 0})^n \setminus \{(0, \ldots, 0)\}$. }

\begin{remark}
One can derive from Equation (\ref{prod1}) that all but one of the exponents $k_1, \ldots k_n$ are equal to zero (consider the form of  $a_d$ and the units of $A(n,d,\textbf{q})$). We leave the statement and justification for this implication until Step 2, where it follows clearly from a rewriting of the equation under examination. 
\end{remark}

\subsection{Step 2}
\textit{In Step 2 we will determine precisely, for $i\in R$, the action of a homomorphism on $\widetilde{x_{i}}$ and $\widetilde{y_{i}}$. We also show that $$\gamma_i^d=\widetilde{q_{i}}^{-\tau_i d}.$$}

Using the action of $\psi$ on $\widetilde{h_{i}}$ we found in Step 1 we now rewrite Equation (\ref{prod1}) as     
\begin{equation}
\begin{aligned}\label{prod2}
U_i(h_{1},\ldots,h_{n}) \left( \prod_{s=1}^{e}\prod_{l=1}^{k_s} \left( q_{s}^{ld}  h_{s}^d-1\right) \right)& \left(\prod_{s=e+1}^{n}\prod_{l=0}^{k_s-1}\left( q_{s}^{-ld}  h_{s}^d-1\right) \right)\\
&=a_d (\widetilde{q_{i}} \gamma_i h_{j}^{m_{i}})=(\widetilde{q_{i}} \gamma_i h_{j}^{m_{i}})^d-1.
\end{aligned}
\end{equation}
Since the factors in the product of the left hand side of Equation (\ref{prod2}) are not invertible (discounting $U_i(h_{1},\ldots,h_{n})$), comparing coefficients shows that $k_j$ is the only nonzero entry in $\textbf{k}$. We can also conclude that $U_i(h_{1},\ldots,h_{n})$ is a Laurent polynomial in $h_{j}$ only and write $U_i(h_{1},\ldots,h_{n})=U_i(h_{j})$ to reflect this.

For simplicity we introduce notation to distinguish between the following two cases: Let $\tau_i=0$ if $\textbf{k}=(0,\ldots,k_{j},\ldots,0)$ and $\tau_i=1$ if $\textbf{k}=(0,\ldots,-k_{j},\ldots,0),$ for $k_{j}>0.$ We can now write Equation (\ref{prod2}) as      
\begin{equation}\label{roots3ee}
U_i(h_{j})\prod_{l=1-\tau_i}^{k_{j}-\tau_i} \left (  q_{j}^{(-1)^{\tau_i}l d}  h_{j}^d-1 \right )=\widetilde{q_{i}}^d \gamma_i^d h_{j}^{m_{i}d} -1.
\end{equation}
We will now prove by contradiction that $k_j=1$. Assuming $k_{j}>1$ we find that $q_{j}^{-(-1)^{\tau_i}(1-\tau_i)}$ and $q_{j}^{-(-1)^{\tau_i}(2-\tau_i)}$ are zeros of the left hand side of Equation (\ref{roots3ee}), substituting these yields
$$\left (q_{j}^{-(-1)^{\tau_i}(1-\tau_i)} \right)^{m_{i}d}= \widetilde{q_{i}}^{-d}\gamma_i^{-d}=\left (q_{j}^{-(-1)^{\tau_i}(2-\tau_i)} \right )^{m_{i}d},$$ implying, by simple manipulation, that $q_{j}^{m_{i}d}=1$. Applying $\psi$ to the relation $\widetilde{x_{i}} \widetilde{h_{i}}=\widetilde{q_{i}} \widetilde{h_{i}} \widetilde{x_{i}}$ gives us $$P_{i}(h_{1},\ldots,h_{n})x_{j}^{k_{j}(1-\tau_i)}y_{j}^{k_j\tau_i}\gamma_i h_{j}^{m_{i}}=\widetilde{q_{i}}  \gamma_i h_{j}^{m_{i}}P_{i}(h_{1},\ldots,h_{n})x_{j}^{k_{j}(1-\tau_i)}y_{j}^{k_j\tau_i}.$$ Simple manipulation indicates that
\begin{equation}\label{ete}
q_{j}^{(-1)^{\tau_i}m_{i} k_{j}}=\widetilde{q_{i}}. 
\end{equation}
Equation (\ref{ete}) implies that $q_{j}^{(-1)^{\tau_i}m_{i} k_{j}d}=\widetilde{q_{i}}^d$, and by substituting $q_{j}^{m_{i}d}=1$, we find that $\widetilde{q_{i}}^d=1$ which contradicts our assumptions and thus, $k_j=1$.

Note, since the derivation of Equation (\ref{ete}) did not rely on the assumption that $k_{j}>1$, we have, by substituting $k_{j}=1$,
\begin{equation}\label{rw}
q_{j}^{(-1)^{\tau_i}m_{i}}=\widetilde{q_{i}}.
\end{equation}
Substituting $k_{j}=1$ into Equation (\ref{roots3ee}) gives us 
\begin{equation}\label{roots3ee2}
U_i(h_{j}) \left ( q_{j}^{(-1)^{\tau_i}(1-\tau_i) d}  h_{j}^d-1 \right )= \widetilde{q_{i}}^d \gamma_i^d h_{j}^{m_{i}d} -1.
\end{equation}
Evaluating $h_{j}$ at $q_{j}^{-(-1)^{\tau_i}(1-\tau_i)}$ in Equation (\ref{roots3ee2}) and using Equation (\ref{rw}) we can conclude that
\begin{equation}\label{lhg12}
\gamma_i^d=\widetilde{q_{i}}^{-\tau_i d}.
\end{equation}
Finally, since $U_i(h_{j})$ is a Laurent polynomial in $h_{j}$ we have
$$P_{i}(h_{1},\ldots,h_{n})= p_i(h_{j}) b_i h_{1}^{t_{i,1}}\cdots h_{n}^{t_{i,n}}~~\mathrm{and}~~P'_{i}(h_{1},\ldots,h_{n})=  p_i'(h_{j}) b_i^{-1} h_{1}^{-t_{i,1}}\cdots h_{n}^{-t_{i,n}}$$ where $p_i(h_{j}) p_i'(h_{j})=U_i(h_{j}),b_{i}\in\mathbb{K^*}$ and $t_{i,1},\ldots,t_{i,n} \in\mathbb{Z}$.  
\\

\textit{
To summarize, in Step 2 we have shown that there exist $(\tau_1,\ldots,\tau_r)\in\{0,1\}^r$, $(m_{1},\ldots,m_r)\in(\mathbb{Z}^*)^r$, $(t_{i,l})\in \mathcal{M}_{r,n}(\mathbb{Z})$, $(b_1,\ldots,b_r),~(\gamma_1,\ldots,\gamma_r)\in(\Bbbk^*)^r$ and $(p_1,\ldots,p_r), (p_1',\ldots,p_r')\in(\Bbbk[h^{\pm 1}])^r$ such that, for all $i \in R$,   
\begin{align*}
 \psi(\widetilde{h_{i}})=\gamma_{i} h_{j}^{m_{i}},~\psi(\widetilde{x_{i}})=& p_i(h_{j}) b_i h_{1}^{t_{i,1}}\cdots h_{n}^{t_{i,n}} x_{j}^{(1-\tau_i)}y_{j}^{\tau_i}\\
&\mathrm{and}~~\psi(\widetilde{y_{i}})=x_{j}^{\tau_i}y_{j}^{(1-\tau_i)}p_i'(h_{j}) b_i^{-1} h_{1}^{-t_{i,1}}\cdots h_{n}^{-t_{i,n}},
\end{align*}
and $\gamma_i^d=\widetilde{q_{i}}^{-\tau_i d}$.}

\subsection{Step 3} 
\textit{In Step 3 we will show that the map $w:R\rightarrow N$ from Step 1 is a partial permutation. We will also derive the necessary condition 
$$q_{w(i)}^{t_{l,w(i)} (1-\tau_i)} q_{w(i)}^{-t_{l,w(i)}\tau_i} q_{w(l)}^{-t_{i,w(l)}(1-\tau_l)} q_{w(l)}^{t_{i,w(l)}\tau_l}=1,$$ for $i,l \in R$, which is required to ensure $\psi$ is consistent on the commutation relations of $\widetilde{A}(r,d,\widetilde{\textbf{q}})$ (see Equation (\ref{comm})).}\\

For simplicity we state the action of $\psi$ on $\widetilde{h_{i}}$ and $\widetilde{h_{e}}$ for $i\neq e\in R$:
\begin{equation*}\label{ONE}
\begin{aligned}
 \psi(\widetilde{h_{i}})=\gamma_{i} h_{j}^{m_{i}},~\psi(\widetilde{x_{i}})=& p_i(h_{j}) b_i h_{1}^{t_{i,1}}\cdots h_{n}^{t_{i,n}} x_{j}^{(1-\tau_i)}y_{j}^{\tau_i}\\
&\mathrm{and}~~\psi(\widetilde{y_{i}})=x_{j}^{\tau_i}y_{j}^{(1-\tau_i)}p_i'(h_{j}) b_i^{-1} h_{1}^{-t_{i,1}}\cdots h_{n}^{-t_{i,n}},
\end{aligned}
\end{equation*}
and
\begin{equation*}\label{TWO}
\begin{aligned}
 \psi(\widetilde{h_{e}})=\gamma_{e} h_{k}^{m_{e}},~\psi(\widetilde{x_{e}})=& p_e(h_{k}) b_e h_{1}^{t_{e,1}}\cdots h_{n}^{t_{e,n}} x_{k}^{(1-\tau_e)}y_{k}^{\tau_e}\\
&\mathrm{and}~~\psi(\widetilde{y_{e}})=x_{k}^{\tau_e}y_{k}^{(1-\tau_e)}p_e'(h_{k}) b_e^{-1} h_{1}^{-t_{e,1}}\cdots h_{n}^{-t_{e,n}},
\end{aligned}
\end{equation*}
where for simplicity we set $j:=w(i)$ and $k:=w(e)$. 

First we prove, by contradiction, that $w$ is a partial permutation. Assume $j=k$. Consider when $\tau_i=\tau_e$ (due to the similarity in the calculation we leave the $\tau_i\neq \tau_e$ to the reader (see Remark \ref{resa})). Applying $\psi$ to the relation $\widetilde{x_{i}}\widetilde{y_{e}}=\widetilde{y_{e}}\widetilde{x_{i}}$ yields
\begin{equation}\label{AEs}
\begin{aligned}
 p_i(h_{j}) b_i & h_{1}^{t_{i,1}}\cdots h_{n}^{t_{i,n}}  x_{j}^{(1-\tau_i)}  y_{j}^{\tau_i} x_{j}^{\tau_e} y_{j}^{(1-\tau_e)} p_e'(h_{j}) b_e^{-1} h_{1}^{-t_{e,1}}\cdots h_{n}^{-t_{e,n}}\\
&=x_{j}^{\tau_e}y_{j}^{(1-\tau_e)}p_e'(h_{j}) b_e^{-1} h_{1}^{-t_{e,1}}\cdots h_{n}^{-t_{e,n}}p_i(h_{j}) b_i h_{1}^{t_{i,1}}\cdots h_{n}^{t_{i,n}} x_{j}^{(1-\tau_i)}y_{j}^{\tau_i}. 
\end{aligned}
\end{equation}
Rearrangement of Equation (\ref{AEs}) gives us 
\begin{equation}\label{rw1}
p_i(h_{j})p_e'(h_{j})a_d(q_{j}^{(1-\tau_i)} h_{j})=P(h_j) a_d(q_{j}^{\tau_i}h_{j}).
\end{equation}
where $P(h_j)\in \Bbbk[h_j^{\pm 1}]$. Evaluating Equation (\ref{rw1}) at $h_{j}=q_j^{-\tau_i}$ yields 
$$ p_i(q_j^{-\tau_i})p_e'(q_j^{-\tau_i})a_d(q_j^{1-2\tau_i})=0$$
which implies, since $a_d(q_j^{1-2\tau_i})\neq 0$, that $p_i(q_j^{-\tau_i})=0$ or $p_e'(q_j^{-\tau_i})=0$. Assuming $p_i(q_j^{-\tau_i})=0$ and noting that $p_i(h_{j}) p_i'(h_{j})=U_i(h_{j})$ we get that $U_i(q_j^{-\tau_i})=0$. By evaluating $h_{j}=q_j^{-\tau_i}$ in Equation (\ref{roots3ee2}) we get
$$ \widetilde{q_{i}}^d \gamma_i^d \left(q_j^{-\tau_i}\right)^{m_{i} d}-1=0$$
and by substituting for $\gamma_i^d$ using Equation (\ref{lhg12}) we get
\begin{equation}
\widetilde{q_{i}}^{(1-\tau_i)d}  \left(q_j^{-\tau_i}\right)^{m_{i}d}=1.
\end{equation}
By considering the cases where $\tau_i=0$ and $\tau_i=1$ seperately and using Equation (\ref{rw}) we derive the contradiction $\widetilde{q_i}^d=1$. The case where $p_e'(q_j^{-\tau_i})=0$ follows in exactly the same way. 
\begin{remark}\label{resa}
The case where $\tau_i\neq \tau_e$ differs only insofar as we apply $\psi$ to the relation $\widetilde{x_{i}}\widetilde{x_{e}}=\widetilde{x_{e}}\widetilde{x_{i}}$ to derive our desired contradiction and show that $w$ is a partial permutation.    
\end{remark}
 
Since $j=w(i)\neq k=w(e)$ for all $i\neq e\in R$, the map $w:R\rightarrow N$ is a partial permutation and we have for all $i\in R$ 
\begin{align*}
\psi(\widetilde{h_{i}})=\gamma_i h_{w(i)}^{m_{i}},~\psi(\widetilde{x_{i}})=& P_{i}(h_{w(i)}) h_{1}^{t_{i,1}}\cdots h_{n}^{t_{i,n}} x_{w(i)}^{(1-\tau_i)}y_{w(i)}^{\tau_i}\\
&\mathrm{and}~~\psi(\widetilde{y_{i}})=x_{w(i)}^{\tau_i}y_{w(i)}^{(1-\tau_i)} P_{i}'(h_{w(i)})h_{1}^{-t_{i,1}}\cdots h_{n}^{-t_{i,n}}.
\end{align*}
Finally applying $\psi$ to the relation $\widetilde{x_{i}}\widetilde{x_{l}}=\widetilde{x_{l}}\widetilde{x_{i}}$ (see the commutation relations $(\ref{comm}))$ yields the relation 
$$q_{w(i)}^{t_{l,w(i)} (1-\tau_i)} q_{w(i)}^{-t_{l,w(i)}\tau_i} q_{w(l)}^{-t_{i,w(l)}(1-\tau_l)} q_{w(l)}^{t_{i,w(l)}\tau_l}=1$$ as required.
\\

\textit{
We have completed the proof of part 1 of Theorem \ref{corol1}.}

\subsection{Step 4} 
\textit{
In Step 4 we will show that $\psi_{\textbf{o}}$ defines a homomorphism between $\widetilde{A}(r,d,\widetilde{\emph{\textbf{q}}})$ and $A(n,d,\emph{\textbf{q}})$.  }\\

It suffices to show that $\psi_{\textbf{o}}$ is consistent on the defining relations of $\widetilde{A}(r,d,\widetilde{\textbf{q}})$. For simplicity we set $\psi_{\textbf{o}}:=\psi$. Thus, 

\begin{align*}
\psi(\widetilde{x_{i}})\psi(\widetilde{h_{i}}) =& p_{i}(h_{w(i)})b_i h_{1}^{t_{i,1}}\cdots h_{n}^{t_{i,n}} x_{w(i)}^{(1-\tau_i)}y_{w(i)}^{\tau_i} \gamma_i h_{w(i)}^{m_{w(i)}}\\
=& q_{w(i)}^{m_{w(i)}((1-\tau_i)-\tau_i)} \psi(\widetilde{x_{i}})\psi(\widetilde{h_{i}}).
\end{align*}
By hypothesis we have $q_{w(i)}^{(-1)^{\tau_i} m_{i}}=\widetilde{q_{i}}$ which gives the desired result that $$\psi(\widetilde{x_{i}})\psi(\widetilde{h_{i}})= \widetilde{q_{i}} \psi(\widetilde{x_{i}})\psi(\widetilde{h_{i}}).$$

Next consider
\begin{align}\label{blak}
\psi(\widetilde{y_{i}}) \psi(\widetilde{x_{i}}) &=x_{w(i)}^{\tau_i}y_{w(i)}^{(1-\tau_i)} p_{i}'(h_{w(i)})p_{i}(h_{w(i)})  x_{w(i)}^{(1-\tau_i)}y_{w(i)}^{\tau_i}\nonumber\\
&=p_i(q_{w(i)}^{\tau_i-(1-\tau_i)}h_{w(i)})p_i'(q_{w(i)}^{\tau_i-(1-\tau_i)}h_{w(i)})a_d(q_{w(i)}^{\tau_i}h_{w(i)}).
\end{align} 
By hypothesis we have the equality
\begin{equation}\label{dfd}
p_i(h_{w(i)})p_i'(h_{w(i)})a_d(q_{w(i)}^{(1-\tau_i)} h_{w(i)})=a_d(\widetilde{q_{i}}\gamma_i h_{w(i)}^{m_{i}}).
\end{equation}
Substituting $q_{w(i)}^{\tau_i-(1-\tau_i)}h_{w(i)}$ into Equation (\ref{dfd}) gives us 
\begin{equation*}
p_i(q_{w(i)}^{\tau_i-(1-\tau_i)}h_{w(i)})p_i'(q_{w(i)}^{\tau_i-(1-\tau_i)}h_{w(i)})a_d(q_{w(i)}^{\tau_i} h_{w(i)})=a_d(\gamma_i h_{w(i)}^{m_{i}}).
\end{equation*}
which in combination with Equation (\ref{blak}) yields the desired result that $$\psi(\widetilde{y_{i}}) \psi(\widetilde{x_{i}})=\psi(a_d(\widetilde{h_i})).$$

Similarly consider
\begin{align*}
\psi(\widetilde{x_{i}})\psi(\widetilde{y_{i}})&=p_{i}(h_{w(i)}) x_{w(i)}^{(1-\tau_i)}y_{w(i)}^{\tau_i}x_{w(i)}^{\tau_i}y_{w(i)}^{(1-\tau_i)} p_{i}'(h_{w(i)})\\
&=p_i(h_{w(i)})p_i'(h_{w(i)})x_{w(i)}^{(1-\tau_i)}y_{w(i)}^{\tau_i}x_{w(i)}^{\tau_i}y_{w(i)}^{(1-\tau_i)}\\
&=p_i(h_{w(i)})p_i'(h_{w(i)})a_d(q_{w(i)}^{(1-\tau_i)}h_{w(i)})
\end{align*}
which by the hypothesis stated in Equation (\ref{dfd}) gives
$$\psi(\widetilde{x_{i}})\psi(\widetilde{y_{i}})=a_d(\widetilde{q_{i}}\gamma_i h_{w(i)}^{m_{i}})=\psi(a(\widetilde{q_{i}}\widetilde{h_{i}})).$$

Since the images of $\widetilde{h}_{i}$ and $\widetilde{h}_{l}$ commute (see Equation (\ref{comm})) it is clear that $\psi$ is consistent on the relation $\widetilde{h_{i}}\widetilde{h}_{l}=\widetilde{h}_{l}\widetilde{h_{i}}$. For the same reason, $\psi$ is consistent on the relations $\widetilde{h_{i}}\widetilde{x}_{l}=\widetilde{x}_{l}\widetilde{h_{i}}$ and $\widetilde{h_{i}}\widetilde{y}_{l}=\widetilde{y_{l}}\widetilde{h_{i}}$.

Finally for $i\neq l$, consider 
\begin{align*}
\psi(\widetilde{x_{i}}) \psi(\widetilde{x}_{l})=p_{i}(h_{w(i)})b_i h_{1}^{t_{i,1}}\cdots h_{n}^{t_{i,n}} x_{w(i)}^{(1-\tau_i)}y_{w(i)}^{\tau_i}p_{l}(h_{w(l)})b_l h_{1}^{t_{l,1}}\cdots h_{n}^{t_{l,n}} \widetilde{x}_{w(l)}^{(1-\tau_l)}y_{w(l)}^{\tau_l}
\end{align*}
which after rearrangement and application of the hypothesis $$q_{w(i)}^{t_{l,w(i)} (1-\tau_i)} q_{w(i)}^{-t_{l,w(i)}\tau_i} q_{w(l)}^{-t_{i,w(l)}(1-\tau_l)} q_{w(l)}^{t_{i,w(l)}\tau_l}=1$$ gives us $\psi(\widetilde{x_{i}}) \psi(\widetilde{x_{l}})= \psi(\widetilde{x_{l}})\psi(\widetilde{x_{i}})$. Similarly $\psi(\widetilde{y_{i}}) \psi(\widetilde{y_{l}})= \psi(\widetilde{y_{l}})\psi(\widetilde{y_{i}})$ and $\psi(\widetilde{x_{i}}) \psi(\widetilde{y_{l}})= \psi(\widetilde{y_{l}})\psi(\widetilde{x_{i}})$. By universal property the algebra $\widetilde{A}(r,d,\widetilde{\textbf{q}})$, the map $\psi$ defines an homomorphism from $\widetilde{A}(r,d,\widetilde{\textbf{q}})$ to $A(n,d,\textbf{q}).$   
\\

\textit{
We have completed the proof of part 2 of Theorem \ref{corol1}.}
\\

To conclude this section, we give the general form of an endomorphism of $A(n,d,\textbf{q})$ subject to the technical assumptions in the statements of Theorem \ref{Them1} and Theorem \ref{Them2k}. Recall that $\sigma_i$ is the automorphism of $\Bbbk[h_1^{\pm 1},\ldots, h_n^{\pm 1}]$ defined by $\sigma_i(h_i)=qh_i$, and $\sigma_i(h_j)=h_j$ for $j\neq i$. 
\begin{corollary}\label{qdaa}
Let $\emph{\textbf{q}}=(q,\ldots, q)$ for $q\in\Bbbk^*$ a non root of unity. Then every endomorphism of $A(n,d,\emph{\textbf{q}})$ is of the form:
 \begin{align*}
\psi(h_{i})=\gamma_i h_{w(i)}^{(-1)^{\tau_i}},~~&\psi(x_{i})=e_i x_{w(i)}^{(1-\tau_i)}y_{w(i)}^{\tau_i}, ~\mathrm{and}~~\psi(y_{i})=x_{w(i)}^{\tau_i}y_{w(i)}^{(1-\tau_i)} e_i',
\end{align*} 
where $w$ is a permutation of $N$, $(\tau_1,\ldots,\tau_n)\in\{0,1\}^n$, $(\gamma_1,\ldots,\gamma_n)\in(\Bbbk^*)^n$ such that $\gamma_i^d=q^{-\tau_i d},$ and $e_i,e_i'$ are units of $A(n,d,\emph{\textbf{q}})$, such that $e_{i} e_{i}'=(-1)^{\tau_i}h_{w(i)}^{-d \tau_i}$ and 
\begin{equation}\label{efjh}
e_i \sigma_{w(i)}^{1-2\tau_i}(e_l)=e_l \sigma_{w(l)}^{1-2 \tau_l}(e_i),
\end{equation}
 for all $i\neq l\in N$.
\end{corollary}
\begin{remark}
By a simple calculation, we can see that when $e_i:=p_{i}(h_{w(i)})b_i h_{1}^{t_{i,1}}\cdots h_{n}^{t_{i,n}}$ and $e_i':=p_{i}'(h_{w(i)})b_i^{-1} h_{1}^{-t_{i,1}}\cdots h_{n}^{-t_{i,n}}$ (as in the statement of Theorem \ref{corol1}), Equation (\ref{efjh}) is equivalent to Equation (\ref{theoeq2}).  

\end{remark}
\begin{proof}
Let $\psi$ be an endomorphism of $A(n,d,\textbf{q})$. By Theorem \ref{corol1} the endomorphism $\psi$ acts on the generators of $A(n,d,\textbf{q})$ as follows:
 \begin{align*}
\psi(h_{i})=\gamma_i h_{w(i)}^{m_{i}},~~&\psi(x_{i})=p_{i}(h_{w(i)})b_i h_{1}^{t_{i,1}}\cdots h_{n}^{t_{i,n}} x_{w(i)}^{(1-\tau_i)}y_{w(i)}^{\tau_i},\\
&\mathrm{and}~~\psi(y_{i})=x_{w(i)}^{\tau_i}y_{w(i)}^{(1-\tau_i)} p_{i}'(h_{w(i)}) b_i^{-1}h_{1}^{-t_{i,1}}\cdots h_{n}^{-t_{i,n}},
\end{align*} 
where the parameters $w, \gamma_i, m_i, t_{i,j}, \tau_i, p_{i}(h_{w(i)}), p_{i}'(h_{w(i)}),$ and $b_i$ are as in the statement of Theorem \ref{corol1} and therefore satisfy Equations (\ref{theoeq1}), (\ref{theoeq2}) and (\ref{theoeq3}). By Equation (\ref{theoeq1}), and since $q$ is not a root of unity, we have that $m_{i}=(-1)^{\tau_i}$ for all $i\in N$. By substituting for $m_i$ in Equation (\ref{theoeq3}) and comparing coefficients of $h_{w(i)}$, we find that $p_i( h_{w(i)})$ and $p_i'(h_{w(i)})$ are monomials in $h_{w(i)}$ such that $p_i( h_{w(i)})p_i'(h_{w(i)})=(-h_{w(i)}^{-d})^{\tau_i}$. Setting $e_i:= p_{i}(h_{w(i)})b_i h_{1}^{t_{i,1}}\cdots h_{n}^{t_{i,n}}$ and $e_i':= p_{i}'(h_{w(i)}) b_i^{-1}h_{1}^{-t_{i,1}}\cdots h_{n}^{-t_{i,n}}$ it is clear $e_{i} e_{i}'=(-h_{w(i)}^{-d})^{\tau_i}$. For simplicity we state an updated form of an endomorphism of $A(n,d,\textbf{q})$:
$$\psi(h_i)=\gamma_i h_{w(i)}^{(-1)^{\tau_i}},~~\psi(x_i)=e_i x_{w(i)}^{(1-\tau_i)}y_{w(i)}^{\tau_i}, ~\mathrm{and}~~\psi(y_i)=x_{w(i)}^{\tau_i}y_{w(i)}^{(1-\tau_i)}e_i'.$$
Finally, Equation (\ref{theoeq2}) is equivalent to the relation $e_i \sigma_{w(i)}^{1-2\tau_i}(e_l)=e_l \sigma_{w(l)}^{1-2 \tau_l}(e_i),$ for all $i\neq l\in N$. This is easily seen by applying $\psi$ to the relation $x_i x_l=x_l x_i$.  
\end{proof}

\section{A Quantum Dixmier Analogue}

Since the algebras $\mathscr{A}_{\alpha,q}\otimes\cdots \otimes \mathscr{A}_{\alpha,q} $ and $\mathscr{H}_{q}^t\otimes\cdots \otimes \mathscr{H}_{q}^t$ are isomorphic to $A(n,1,\textbf{q})$ and $A(n,2t,\textbf{q})$ respectively (see Section \ref{Prel}), Theorems \ref{Them1} and \ref{Them2k} are specializations of the following corollary to Theorem \ref{corol1} (see Remark \ref{rmk:utile} as to why this is sufficient).

\begin{proposition}\label{cor2b}
Let $\emph{\textbf{q}}=(q,\ldots, q)$ for $q\in\Bbbk^*$ a non root of unity. Then, every endomorphism of $A(n,d,\emph{\textbf{q}})$ is an automorphism.   
\end{proposition}
\begin{proof}
Let $\psi$ be defined as in the statement of Corollary $\ref{qdaa}$. We will construct a candidate inverse of $\psi$, say $\phi$, and show that $\phi$ is a endomorphism of $A(n,d,\textbf{q})$. It is clear that $\phi(h_{w(i)})=\gamma_i^{-(-1)^{\tau_i}}h_{i}^{(-1)^{\tau_i}}$ is a well defined automorphism when restricted to $\Bbbk[h_1^{\pm 1},\ldots, h_n^{\pm 1}]$. We propose the following candidate inverse of $\psi$:
\begin{align*}
\phi(h_{w(i)})=\gamma_i^{-(-1)^{\tau_i}}h_{i}^{(-1)^{\tau_i}},~~&\phi(x_{w(i)})=\phi(e_i )^{-(1-\tau_i)} \sigma_i^{-1}(\phi(e_i')^{-\tau_i})x_{i}^{1-\tau_i}y_{i}^{\tau_i},\\
&\mathrm{and}~~\phi(y_{w(i)})=x_{i}^{\tau_i}y_{i}^{1-\tau_i}\sigma_i^{-1}(\phi(e_i )^{-\tau_i})\phi(e_i')^{-(1-\tau_i)}.
\end{align*}   

We will now show that $\phi$ is a well defined endomorphism of $A(n,d,\textbf{q})$ by checking the conditions of Corollary \ref{qdaa}. Since $\gamma_i^d=q^{-\tau_i d}$, a brief computation shows that $(\gamma_i^{-(-1)^{\tau_i}})^d=q^{-\tau_i d}$. Next we show that  
\begin{equation}\label{odsa}
\left( \phi(e_i )^{-(1-\tau_i)} \sigma_i^{-1}(\phi(e_i')^{-\tau_i}) \right) \left( \sigma_i^{-1}(\phi(e_i )^{-\tau_i})\phi(e_i')^{-(1-\tau_i)} \right)=(-h_i^{-d})^{\tau_i}.
\end{equation}
By rearranging the left hand side of Equation (\ref{odsa}) we get $\phi(e_i e_i')^{-(1-\tau_i)} \sigma_i^{-1}(\phi(e_i e_i')^{-\tau_i}),$ which by the substitution $e_i e_i'=(-h_{w(i)}^{-d})^{\tau_i}$ gives 
\begin{equation}\label{pfp}
\phi(-h_{w(i)}^{-d})^{-(1-\tau_i)\tau_i} \sigma_i^{-1}(\phi(-h_{w(i)}^{-d})^{-\tau_i^2}).
\end{equation} 
It is easy to see that when $\tau_i=0,$ Equation (\ref{pfp}) is equal to 1 as required. We set $\tau_i=1$ in Equation (\ref{pfp}) and find  
\begin{equation}
\sigma_i^{-1}(\phi(-h_{w(i)}^{-d})^{-1})=-\sigma_i^{-1}(\gamma_i h_i^{(-1)})^{d}=-q^{-d}\gamma_i^{d}h_i^{d}=-h_i^d
\end{equation} 
as required (note the last step follows by the substitution $\gamma_i^d=q^{-d}$).

Next we show that 
 \begin{equation}\label{equl1}
 \phi(x_{w(i)}) \phi(x_{w(l)}) = \phi(x_{w(l)})\phi(x_{w(i)}).
 \end{equation} 
At this point we return to the notation of Theorem \ref{corol1} and precisely express the units $e_i$ and $e_i'$. We set $e_i:=p_i(h_{w(i)}) h_{w(1)}^{t_{i,w(1)}} \cdots h_{w(n)}^{t_{i,w(n)}}$ and $e_i':=p_i'(h_{w(i)}) h_{w(1)}^{-t_{i,w(1)}} \cdots h_{w(n)}^{-t_{i,w(n)}}$, where $$p_i(h_{w(i)})p'_i(h_{w(i)})=(-h_{w(i)}^{-d})^{\tau_i}$$ and $p_i(h_{w(i)}), p'_i(h_{w(i)})$ are monomials in $h_{w(i)},$
and 
\begin{equation}\label{eq6y}
q^{t_{l,w(i)} (1-\tau_i)} q^{-t_{l,w(i)}\tau_i} q^{-t_{i,w(l)}(1-\tau_l)} q^{t_{i,w(l)}\tau_l}=1,
\end{equation}
which is an equivalent condition to Equation (\ref{efjh}). 
For simplicity, we make the following observation regarding the way $\phi(x_{w(i)})$ and $\phi(x_{w(l)})$ commute. To show that Equation (\ref{equl1}) holds, it is clear that we need only consider the coefficients that appear as the $h_l$ component of $\phi(x_{w(i)})$ passes the $x_l,y_l$ terms in $\phi(x_{w(l)})$ and as the $h_i$ component of $\phi(x_{w(l)})$ passes $x_i,y_i$ terms in $\phi(x_{w(i)})$. We reflect this observation in our notation by representing all of the unnecessary information by ellipses.
We highlight that $(-1)^{\tau_i}=(1-2\tau_i)$. Thus,
  \begin{eqnarray*}
  \phi  (x_{w(i)})  \phi(x_{w(l)}) &= &\left( \cdots  h_l^{(\tau_i-1)(-1)^{\tau_i}t_{i,w(l)}} \cdots h_l^{\tau_i (-1)^{\tau_i}t_{i,w(l)}}\cdots x_i^{1-\tau_i}y_i^{\tau_i} \right)\\
  & & \times \left( \cdots  h_i^{(\tau_l-1)(-1)^{\tau_l} t_{l,w(i)}}  \cdots h_i^{\tau_l(-1)^{\tau_l}t_{l,w(i)}}\cdots x_l^{1-\tau_l}y_l^{\tau_l}\right)\\
 &=  &\left(q^{(2\tau_l-1)(2\tau_i-1)(-1)^{\tau_i}t_{i,w(l)}}\right) \left( q^{(1-2\tau_i)(2\tau_l-1)(-1)^{\tau_l}t_{l,w(i)}}\right) \phi(x_{w(l)})\phi(x_{w(i)})\\
  & = & \left(q^{(-1)^{\tau_i}t_{i,w(l)}-(-1)^{\tau_l}t_{l,w(i)}}\right)^{(2\tau_l-1)(2\tau_i-1)} \phi(x_{w(l)})\phi(x_{w(i)})\\
  & = & \left(q^{(1-2\tau_i)t_{i,w(l)}-(1-2\tau_l)t_{l,w(i)}}\right)^{(2\tau_l-1)(2\tau_i-1)} \phi(x_{w(l)})\phi(x_{w(i)})
  \end{eqnarray*}
  which in combination with Equation (\ref{eq6y}) gives the desired result. Note that it is easier to apply Equation (\ref{eq6y}) if we consider the choices of $\tau_i$ and $\tau_l$ separately. We leave to the reader the calculations to show that $\phi$ is consistent on the remaining relations (see Equation (\ref{comm})). These follow in a similar way. We have shown that $\phi$ conforms to the necessary conditions from Theorem \ref{corol1} to be an endomorphism of $A(n,d,\textbf{q})$. By direct computation we can see $\psi \phi=\phi \psi=\mathrm{id}$.
\end{proof}

We will now offer a counter example to show that our quantum Dixmier analogue is false when $q$ is a root of unity.
\begin{corollary}\label{endi}
There exist non-invertible endomorphisms of $A(n,d,\emph{\textbf{q}})$ when (at least) one coordinate $q:=q_i$ of $\emph{\textbf{q}}$ is a root of unity. 
\end{corollary} 
\begin{proof}
Let $q$ be a $t^{\mathrm{th}}$ root of unity. It is enough to find an example of a non-invertible endomorphism of $\Bbbk[h^{\pm 1}](\sigma,h^d-1)$ where $\sigma(h)=qh$. Define the polynomial $U(h)=\displaystyle\sum^{t}_{l=0} u_l h^{dl}$, where $u_i=q^{id}$ for $0\leq i \leq t$ so that 
$$ U(h)((qh)^d-1)=(qh)^{d(t+1)}-1.$$
Then it follows from Theorem \ref{corol1} that we define an endomorphism $\psi$ of $\Bbbk[h^{\pm 1}](\sigma,h^d-1)$ by setting
$$\psi(h)=h^{t+1},~~\psi(x)=U(h)x, ~\mathrm{and}~~\psi(y)=y.$$ 
Since by assumption $t>1$ we can see that $\psi$ is not invertible by considering the action on $h$. By taking a tensor product with $n-1$ copies of the identity, we can lift $\psi$ to an non-invertible endomorphism of $A(n,d,\textbf{q})$.
\end{proof}

\section{A Quantum Tame Generators Problem}\label{qtgp}

For the entirety of this section let $\textbf{q}=(q,\ldots,q)$ for $q\in\Bbbk^*\setminus\{z|z^d= 1\}$. Also,  recall from Section $\ref{Prel}$ that $A(1,d,q)\simeq \Bbbk[h^{\pm 1}](\sigma, h^d-1)$ where $\sigma(h)=qh$. Since $A(n,d,\textbf{q})$ has a nontrivial group of units (See Lemma \ref{Unitsl}) we can find automorphisms of $A(n,d,\textbf{q})$ which are not tame. For example consider the automorphism of $A(n,d,\textbf{q})$ defined in the following way 
\begin{equation}\label{unitaut}
h_i\mapsto h_i,~~x_i\mapsto h_ix_i,~~\mathrm{and}~~y_i\mapsto y_i h_i^{-1}.
\end{equation}
Since we are interested in determining whether the complexity of the automorphisms of $A(n,d,\textbf{q})$ fundamentally changes as $n$ increases, we will take inspiration from the traditional definition of tame to define an $A(n,d,\textbf{q})$ specific quantum analogue which we will denote $\mathrm{qwa}$-tame. Before stating the definition of $\mathrm{qwa}$-tame, we will highlight three natural families of automorphisms each of which is inspired by a family of tame automorphisms. 

The first two families arise from the fact that both the polynomials in $n$ variables and the algebra $A(n,d,\textbf{q})$ can be constructed as $n$ tensor copies of $\Bbbk[x]$ and $\Bbbk[h^{\pm 1}](\sigma, h^d-1)$ respectively. By this construction we can pick $g\in \mathrm{Aut}(\Bbbk[h^{\pm 1}](\sigma_q,h^d-1))$ and lift to an automorphism $\phi_{g}:=g \otimes 1 \otimes \cdots \otimes 1$ of $A(n,d,\textbf{q})$. For our second family, we associate to each permutation $w$ of $N$ a (unique) automorphism $\chi_w$ of $A(n,d,\textbf{q})$ defined as follows: 
$$\chi_w(h_i)=h_{w(i)},~~\chi_w(x_i)=x_{w(i)}, ~\mathrm{and}~~\chi_w(y_i)=y_{w(i)}.$$ 
Finally we introduce a family to include automorphisms arising from the non-trivial group of units of $A(n,d,\textbf{q})$ (for instance see Equation (\ref{unitaut})). This family generalizes the scalar automorphisms in the traditional definition. Recall that $\sigma_i\in\mathrm{Aut}(\Bbbk[h_1^{\pm 1},\ldots h_n^{\pm 1}])$ such that $\sigma_i(h_i)=qh_i$, and $\sigma_i(h_j)=h_j$ for $j\neq i$. For a vector of units in $A(n,d,\textbf{q})$, say $\textbf{u}:=(u_1,\ldots,u_n)$, such that $u_i \sigma_i(u_l)=u_l\sigma_l( u_i)$ for $l \neq i$ (note this encodes Equation (\ref{theoeq2})), there exists a (unique) automorphism $\xi_{\textbf{u}}$ of $A(n,d,\textbf{q})$ defined as follows:  
$$\xi_{\textbf{u}}(h_i)=h_{i},~~\xi_{\textbf{u}}(x_i)=u_ix_{i}, ~\mathrm{and}~~\xi_{\textbf{u}}(y_i)=y_{i}u_i^{-1}.$$
\begin{definition}\label{def1}
Let $\psi$ be an automorphism of $A(n,d,\emph{\textbf{q}})$, we say that $\psi$ is $\mathrm{qwa}$-tame if $\psi$ is in the subgroup generated by the families of automorphisms $\phi_g, \chi_w$ and $\xi_{\textbf{u}}$.
\end{definition}
\noindent
To enable us to practically apply Definition \ref{def1} we recall from \cite{Launois2} and \cite{Vivas} the classification of automorphisms of $\Bbbk[h^{\pm 1}](\sigma, h^d-1)$.
\begin{proposition}
Let $\psi$ be an automorphism of $\Bbbk[h^{\pm 1}](\sigma, h^d-1)$. Then $\psi$ is defined on the generators of $\Bbbk[h^{\pm 1}](\sigma, h^d-1)$ in the following way:
$$\psi(h)= \gamma h^{(-1)^{\tau}},~~\psi(x)=ux^{(1-\tau)}y^{\tau}, ~\mathrm{and}~~\psi(y)=y^{(1-\tau)}x^{\tau}u'$$
where $\tau\in\{0,1\}$, $\gamma^d=(q^{-d})^{\tau}$ and $u, u'\in \Bbbk[h^{\pm 1}]$ such that $uu'=(- h^{-d})^{\tau}$.
\end{proposition} 
\noindent

Since the algebras $\mathscr{A}_{\alpha,q}\otimes\cdots \otimes \mathscr{A}_{\alpha,q}$ and $\mathscr{H}_{q}^t\otimes\cdots \otimes \mathscr{H}_{q}^t$ are isomorphic to $A(n,1,\textbf{q})$ and $A(n,2t,\textbf{q})$ respectively (see Section \ref{Prel}), Theorems \ref{A2} and \ref{A3} are specializations of the following corollary to Theorem \ref{corol1}.  
\begin{corollary}\label{Cii}
Every automorphism of $A(n,d,\emph{\textbf{q}})$ is $\mathrm{qwa}$-tame.
\end{corollary}   
\begin{proof}
Let $\psi$ be an automorphism of $A(n,d,\textbf{q})$. By Corollary \ref{qdaa} we have that $\psi$ acts on the generators of $A(n,d,\textbf{q})$ as follows:  
$$\psi(h_i)=\gamma_i h_{w(i)}^{(-1)^{\tau_i}},~~\psi(x_i)=e_i  x_{w(i)}^{(1-\tau_i)}y_{w(i)}^{\tau_i},~~\mathrm{and}~~\psi(y_i)=x_{w(i)}^{\tau_i}y_{w(i)}^{(1-\tau_i)}e_i' $$ 
where the parameters $w, \gamma_i, \tau_i, e_i$ and $e_i'$ are as in the statement of Corollary \ref{qdaa}. By applying $\mathrm{qwa}$-tame automorphisms (see Definition \ref{def1}), we will reduce $\psi$ to an obvious $\mathrm{qwa}$-tame automorphism. Applying the automorphism $\chi_{w^{-1}}$ gives us       
$$\chi_{w^{-1}}\psi(h_i)=\gamma_i h_{i}^{(-1)^{\tau_i}},~~\chi_{w^{-1}}\psi(x_i)=\chi_{w^{-1}}(e_i ) x_{i}^{(1-\tau_i)}y_{i}^{\tau_i}, ~\mathrm{and}~~\chi_{w^{-1}}\psi(y_i)=x_{i}^{\tau_i}y_{i}^{(1-\tau_i)}\chi_{w^{-1}}(e_i' ).$$

Next we fix the notation $\phi_{g_j}^{(1,j)}:=\chi_{(1,j)}\phi_{g_j}\chi_{(1,j)}$, where $g_j$ is the automorphism of $\Bbbk[h^{\pm 1}](\sigma, h^d-1)$ defined by    
$$g_j(h)=\gamma_j^{-(-1)^{\tau_j}} h^{(-1)^{\tau_j}},~~g_j(x)=p_j^{\tau_j}y^{\tau_j}x^{(1-\tau_j)}, ~\mathrm{and}~~g_j(y)=x^{\tau_j}y^{(1-\tau_j)}(p_j')^{\tau_j}$$ 
with $p_j, p_j'\in \Bbbk[h^{\pm 1}]$ such that $p_j p_j'=(- h^{-d})^{\tau_j}$ and $\phi_{g_j}$ is defined as in Definition \ref{def1}. Moreover, we let $$G:=\phi_{g_1}^{(1,1)}\phi_{g_2}^{(1,2)}\cdots\phi_{g_n}^{(1,n)}.$$ Thus, $G$ is a qwa-tame automorphism of $A(n,d,\textbf{q})$ and we have:
$$G(h_i)=\gamma_i^{-(-1)^{\tau_i}} h_i^{(-1)^{\tau_i}},~~G(x_i)=p_i^{\tau_i}y_i^{\tau_i}x_i^{(1-\tau_i)}, ~\mathrm{and}~~G(y_i)=x_{i}^{\tau_i}y_{i}^{(1-\tau_i)}(p_i')^{\tau_i}$$
for all $i\in N$. We can easily check that the action of $G\chi_{w^{-1}}\psi$ on the generators of $A(n,d,\textbf{q})$ is given by:  
\begin{align*}
G\chi_{w^{-1}}\psi(h_i)= h_{i},&~~G\chi_{w^{-1}}\psi(x_i)=G(\chi_{w^{-1}}(e_i)) \sigma_i(p_i')^{\tau_i} x_{i},\\
&~~~~~~~~~\mathrm{and}~~G\chi_{w^{-1}}\psi(y_i)=y_{i}G(\chi_{w^{-1}}(e_i'))\sigma_i^{-1}(p_i)^{\tau_i}.
\end{align*}
Since $G\chi_{w^{-1}}\psi$ is an automorphism of $A(n,d,\textbf{q})$, the units $G(\chi_{w^{-1}}(e_i)) \sigma_i(p_i')^{\tau_i}$ and $G(\chi_{w^{-1}}(e_i'))\sigma_i^{-1}(p_i)^{\tau_i}$ must decompose in the following way:
 $$G(\chi_{w^{-1}}(e_i)) \sigma_i(p_i')^{\tau_i}= U_i~~ \mathrm{and}~~G(\chi_{w^{-1}}(e_i'))\sigma_i^{-1}(p_i)^{\tau_i}= U_i^{-1}$$
 where $U_i$ is a unit of $A(n,d,\textbf{q})$ such that $U_i\sigma_i(U_l)=U_l \sigma_l(U_i)$ for $l \neq i$. Applying the qwa-tame automorphism  
 $\xi_{\textbf{u}}$, where $\textbf{u}:=(U_1^{-1},\ldots,U_n^{-1})$, yields  
 $$\xi_{\textbf{u}}G\chi_{w^{-1}}\psi(h_i)= h_{i},~~\xi_{\textbf{u}}G\chi_{w^{-1}}\psi(x_i)=x_{i},~~\mathrm{and}~~G\chi_{w^{-1}}\psi(y_i)=y_{i}.$$
Thus, $\psi = \chi_w G^{-1} \xi_{\textbf{u}}^{-1}$ is qwa-tame.      
\end{proof}

\section{Future directions}

Following the submission to the arXiv of the preprint to this article, a Dixmier type problem for a quantized Weyl algebra was solved in \cite{Tang} by Tang. The algebra studied, denoted $(\mathcal{A}_n^{\bar{q},\Lambda}(\mathbb{K}))_{\mathcal{Z}}$, is isomorphic to $A(n,1,\textbf{q}),$ when $\textbf{q}=(q_1,\ldots, q_n)$. It was shown, under the condition $q_1^{i_1} q_2^{i_2}\cdots q_n^{i_n}=1$ implies $i_1=i_2=\cdots=i_n=0$, that every endomorphism of $A(n,1,\textbf{q})$ is an automorphism. One can ask if the same result holds for the tensor product of quantum generalized Weyl algebras with a general choice of defining polynomial $a(h)$, namely the algebras
$$\Bbbk[h^{\pm 1}](\sigma_1, a(h_1))\otimes\cdots \otimes \Bbbk[h^{\pm 1}](\sigma_n, a(h_n)),$$
where $n$ and $\sigma_i$ are defined as they have been throughout. Based on the results of this article, and those obtained in \cite{Launois2} and \cite{Tang}, we believe it is possible to show that every endomorphism is an automorphism when $q_1^{i_1} q_2^{i_2}\cdots q_n^{i_n}=1$ implies $i_1=i_2=\cdots=i_n=0$, or when $\textbf{q}=(q,\ldots, q)$ for $q$ not a root of unity. It would be especially nice if these results could be obtained using less computational methods than those seen here.

The techniques used in Section \ref{Secthe} could be applied to produce similar results if the defining polynomials of the generalized Weyl algebras used for each tensor factor contain multiple generators from the base ring. Consider the generalized Weyl algebras $$\displaystyle\bigotimes_{i=1}^n \Bbbk[h_{1}^{\pm 1},\ldots, h_{m_i}^{\pm 1}]\big(\sigma_i, a_i(h_{1}\ldots h_{m_i})\big)$$ where $(m_1,\ldots,m_n)\in (\mathbb{N}^*)^n$. These algebras include the class of simple algebras (called multiparameter Weyl algebras) which were introduced by Benkart in \cite{Benkart}. This class of algebras were also studied in \cite{Futorny}, where rather than generalizing as we have suggested, the authors considered Benkart's algebras as members of a class of twisted generalized Weyl algebras.  The class of algebras considered in \cite{Futorny} should be a fruitful place to apply the ideas and techniques used in our classification. 

For the algebras studied in this article, we believe that the rigidity of the relations means there are very few if any possibilities for locally nilpotent derivations. We have a number of results regarding derivations for these algebras and will address the classification of locally nilpotent derivations in detail in a coming article.

\ \\ \hspace{1cm}
\begin{minipage}[c]{\linewidth}
~\\
\noindent 
A.P. Kitchin \\
E-mail: {\tt andrew.p.kitchin@gmail.com} \\

S. Launois\\ 
School of Mathematics, Statistics \& Actuarial Science,\\ 
University of Kent, \\
Canterbury, Kent CT2 7FS, United Kingdom\\
E-mail: {\tt S.Launois@kent.ac.uk}

\end{minipage}

\end{document}